\documentclass[12pt]{article}
\usepackage{amssymb,amsmath,amsthm}
\usepackage{hyperref}
\usepackage{cite}
\usepackage[english]{babel}
\newtheorem{theorem}{Theorem}
\newtheorem{lemm}{Lemma}

\usepackage{times, url}
\begin{document}
	\setcounter{page}{1} 
	\begin{center}
		{\LARGE \bf  On the largest prime factors of shifted semiprime numbers. }
		\vspace{8mm}
		
		{\Large \bf Tam D. Do$^1$,}
		\vspace{3mm}
		
		$^1$ School of applied mathematics and informatics, Hanoi University of science and technology \\ 
		1 Dai Co Viet, Hanoi, Vietnam \\
		e-mail: \url{tam.doduc@hust.edu.vn}
		\vspace{2mm}
	\end{center}
	\vspace{2 mm}
	\section{Introduction} 
	Let $P^{+}(n)$ be the largest prime factor of $n$. Chebyshev proved that $$\lim_{n\rightarrow{\infty}}P^{+}(n^2+1)/n=\infty.$$ In this direction, C. Hooley \cite{Hooley67} showed that $$P^{+}(n^2+1)/n\geq n^{0.1001483...}$$
	for sufficiently large $n.$ In 1982, J. M. Deshouillers and H. Iwaniec \cite{Iwaniec82} improved this result $$P^{+}(n^2+1)/n\geq n^{0.202468...}.$$ Their improvement came as an application of their bounds for linear forms of Kloosterman sums. This result was improved by Bretèche and Drappeau \cite{Bretèche and Drappeau} and Jori Merikoski \cite{Merikoski}.
	
	There is an interesting analogue of Chebyshev's problem about largest prime factor of shifted primes, which has an obvious connection with conjectures concerning the
	infinitude of prime-pairs. Y. Motohashi \cite{Motohashi} showed that, there is exist a constant $1/2<\theta <1-1/(2\sqrt[4]{e})=0.611059...$ such that $$P^{+}(p+a)>x^{\theta}$$ for all sufficiently large $x$ and fixed integer $a$. His method involved the use of both Bombieri's theorem and the Brun-Titchmarsh theorem about number of primes in arithmetic progression. Applying the Brun-Titchmarsh theorem on average, C. Hooley \cite{Hooley72} proved that parameter $\theta$ can be replaced by $0.619933...$. Subsequently this limit for $\theta$ was improved by C. Hooley \cite{Hooley73} to $0.625$. The main idea in Hooley's proof is applying Selberg's method to certain sequences of the form $kq-a$ in which the positive integer $k$ and the prime $q$ lie in appropriate ranges.
	
	In this present paper, the largest prime factor $P^{+}(n+a)$, where semi-prime $n<x$, is studied. We remark that semi-prime numbers form sufficiently “rare” sub-sequences of the natural series and for $c\geq 2$ numbers $p_1p_2^c$ is close to the sequence of prime numbers. The latter subject obviously has connection with problem about largest prime factor of shifted primes.  We show that $P^{+}(n+a)>x^{
		\theta}$, $0.6<\theta<0.625$ for sufficiently large $x$. We will follow the treatments, which are used in papers of Y. Motohashi \cite{Motohashi} and C. Hooley \cite{Hooley73}.
	\section{Lemmas}
	In order to prove the latter theorem we will need the following lemmas:
	\begin{lemm}[\cite{Hooley73}]
		Let $r$ be constant such that $1<r<2$ and $a$ be given non-zero integer. For $X$ satisfy the inequalities $x^{1/2}<X<x^{3/4}$ we define $$V(X)=V_r(X)=\sum_{X<q\leq rX} \pi(x;q,a)\log{q}.$$ Then, we have $$V(X)<\frac{(4+\eta)(r-1)x}{\log{x}}.$$
	\end{lemm}
	Note that in proof of lemma 2.1 the constants implied by the O-notation depend on given integer $a$ as divisor function $\tau(a)$. Hence, lemma 2.1 still hold for case $a$ increasing like exponent of $x$.
	\begin{lemm}[\cite{Mongomery}, p. 88] For $y \geq 2$, the following inequality holds
		$$
		\sum_{\substack{x<n \leq x+y \\ n - \text{semi-prime}}} 1 \leq \frac{2 y \log \log y}{\log y}\left(1+O\left(\frac{1}{\log \log y}\right)\right)
		$$
	\end{lemm}
	One should find more detailed proof of this lemma in D. Hensley's paper \cite{hensly78}.
	\section{Largest prime factor of shifted semi-primes}
	Our main theorem can be formulated as follows:
	\begin{theorem}
		Let $a$ be fixed integer then for all constants $0.6<\theta <0.625$ and $x>x_0(\theta)$ the greatest prime factor of $$\prod_{\substack{n-\text{semi-prime}\\n\leq x}}(n+a)$$ exceeds $x^{\theta}$.
	\end{theorem}
	\begin{proof}
		For simplicity we confine our attention to the case in which $a=-1$. Beginning the proof we follow Motohashi \cite{Motohashi}. Let $$R=\prod_{\substack{2\leq p_1p_2\leq x\\p_1\leq p_2}}(p_1p_2-1)=\prod_{q\leq P^+(p_1p_2-1)}q^{\nu_q(R)}$$  where the product is taken over primes $p_1, p_2, q$ and $P^+(p_1p_2-1) - $ largest prime factor of $R$. We can easily obtain for the exponent $$\nu_q(R)=m_1+m_2+m_3+\cdots$$ where $m_s=m_s(q)$ is number of prime pairs such that $n-1=p_1p_2-1\equiv 0\pmod{q^s} $. Here sum over $m_s$ is finite $s\leq \tau=\log(x+1)/\log{q}.$ Fundamental calculation yields 
		\begin{equation}
			\label{logR}
			\log(R)=\sum_{\substack{2\leq p_1p_2\leq x\\p_1\leq p_2}}\log(p_1p_2-1)=x\log{\log{x}}+O(x).
		\end{equation} 
		On the other hand, we have
		$$\log(R)=\sum_{q\leq P^+(n-1)}\log{q}\sum_{1\leq s\leq \tau}m_s(q).$$
		Let $Q=\sqrt[4]{x}(\log{x})^{-B}$. According to the Bombieri-Vinogradov theorem, the constant $B$ can be chosen such that the following inequality holds:
		$$
		\sum_{q\leq Q}\max_{(b,q)=1}\left| \pi\left(\frac{x}{p_1};q,b\right)-\frac{Li(x/p_1)}{\varphi(q)} \right| \ll \frac{x}{p_1(\log{x/p_1})^2}.
		$$
		Then we divide the sum over $q\leq p_x$ and $1\leq s\leq \tau $ into 3 parts, namely $A, B$ and $C$ as follows:
		\begin{itemize}
			\item $ q^s\leq Q$;
			\item $s\geq 2, Q<q^s$;
			\item $s=1, Q<q\leq P^{+}(n-1)$.
		\end{itemize}
		
		\noindent Upper bound of $A$ can be easily obtained from the latter inequality. We have
		$$A=\sum_{q^s\leq Q}(\log{q})\sum_{\substack{p_2\equiv \overline{p_1} \pmod{q^s}\\p_1p_2\leq x+1\\p_1\leq p_2}}1=\sum_{q^s\leq Q}(\log{q})\sum_{p_1\leq \sqrt{x}}\frac{Li(x/p_1)}{\varphi(q^s)}+O\left ( \frac{x\log{\log{x}}}{(\log{x})^2}\right )\ $$
		$$=\sum_{q^s\leq Q}\frac{\log{q}}{\varphi{(q^s)}}\left( \frac{x\log{\log{x}}}{\log{x}}+O\left(\frac{x\log{\log{x}}}{(\log{x})^2}\right)\right)+O\left ( \frac{x\log{\log{x}}}{(\log{x})^2}\right ).$$
		Consider the following sum
		$$\sum_{q^s\leq Q}\frac{\log{q}}{\varphi{(q^s)}}=\sum_{q\leq Q}\sum_{s\leq \log{Q}/\log{q}}\frac{\log{q}}{q^{s-1}(q-1)}\leq \sum_{q\leq Q} \frac{\log{q}}{q-1}\sum_{s=1}^{+\infty}\frac{1}{q^{s-1}}=\log{Q}+O(1).$$
		Since $Q=\sqrt[4]{x}(\log{x})^{-B}$, we have
		\begin{equation}
			\label{sum A}
			A=\frac{1}{4}x\log{\log{x}}+O\left(\frac{x(\log{\log{x}})^2}{\log{x}}\right).
		\end{equation}
		In sum $B$ semi-primes $p_1p_2$ lie in arithmetic progression with large difference $q^s$ and $s\geq 2$. Therefore the contribution to $\log{(R)}$ of this term is small. We have
		\begin{equation}
			\label{sum B}
			B \ll x^{11/12}(\log{x})^{1+B/3}.
		\end{equation}
		
		We now turn our attention to the sum $C$. For $0.6<\theta \leq 0.625$ and $0<\varepsilon<0.01 $ arbitrarily small number, we write $$C=\left(\sum_{p_1\leq \sqrt{x}}\sum_{Q<q\leq \sqrt[4]{x}}+\sum_{p_1\leq \sqrt{x}}\sum_{\sqrt[4]{x}<q\leq \sqrt{x/p_1}}+\sum_{p_1\leq x^{1-4\theta/(3-\varepsilon)}}\sum_{\sqrt{x/p_1}<q\leq x^{\theta}}\right.$$
		$$ \left.+\sum_{x^{1-4\theta/(3-\varepsilon)<p_1\leq \sqrt{x}}}\sum_{(x/p1)^{3/4-\varepsilon}<q\leq x^{\theta} }\right)(\log{q})\pi'\left(\frac{x}{p_1};q,\overline{p_1}\right)+\sum_{x^{\theta}<q\leq P^{+}(n-1)}m_1(q)$$
		$=C_1+C_2+C_3+C_4+C_5,$ say. Here symbol $\pi'$ means we count number of primes with condition $p_1\leq p_2\leq x/p_1.$ From \ref{logR}, \ref{sum A}, \ref{sum B} and the latter equation we deduce that if $C_1+C_2+C_3+C_4 \leq (0.75-\beta)$, where $0<\beta$ an arbitrarily small number, then $C_5>0$ i.e., $P^+{(n-1)}$ exceeds $x^{\theta}$. In the summation $C_4$ the threshold $x^{1-4\theta/(3-\varepsilon)}<p_1$ is deduced from condition $x^{\theta}\geq (x/p_1)^{3/4-\varepsilon}$.
		
		The term $C_1$, by Brun-Titchmarsh theorem, does not exceed
		$$
		\sum_{Q<q\leq \sqrt[4]{x}}\log{q}\sum_{p_1\leq \sqrt[4]{x}}\pi(x/p_1;q,\overline{p_1})
		$$
		\begin{equation}
			\label{C_1}
			\ll \sum_{Q<q\leq \sqrt{x+1}}\log{q}\sum_{p_1\leq \sqrt{x}}\frac{x}{p_1\varphi(q)\log{(2x/p_1q)}}\ll \frac{x(\log{\log{x}})^2}{\log{x}}.
		\end{equation}
		
		Let us now consider the summation $C_2$. Put $Q_1=\sqrt{x/p_1}(\log{x/p_1})^{-B},$ where $B=B(p_1)$ is defined from Bombieri-Vinogradov theorem such that
		$$
		\sum_{q\leq Q_1}\max_{(b,q)=1}\left| \pi\left(\frac{x}{p_1};q,b\right)-\frac{Li(x/p_1)}{\varphi(q)} \right| \ll \frac{x}{p_1(\log{x/p_1})^2}.
		$$
		We write
		$$C_2\leq \left(\sum_{p_1\leq \sqrt{x}}\sum_{\sqrt[4]{x}<q\leq Q_1}+\sum_{p_1\leq \sqrt{x}}\sum_{Q<q\leq \sqrt{x/p_1}}\right)(\log{q})\pi\left(\frac{x}{p_1};q,b\right).$$
		It follows from Bombieri-Vinogradov theorem that the contribution of the first term to $C_2$ does not exceed
		$$\sum_{p_1\leq \sqrt{x}}\sum_{\sqrt[4]{x}<q\leq Q_1}\frac{(\log{q})Li(\frac{x}{p_1})}{\varphi(q)}+ O\left(\frac{x}{p_1(\log{\frac{x}{p_1}})^2}\right)\leq 
		\frac{1}{4}x\log{\log{x}}+O(x).$$
		The second term of $C_2$ can be evaluated by Brun-Titchmarsh theorem. Its contribution to $C_2$ is $O(x\log{\log{x}}/\log{x})$. Hence, we have
		\begin{equation}
			\label{C2}
			C_2\leq \frac{1}{4}x\log{\log{x}}+O(x).
		\end{equation}
		
		Next, we study contribution of $C_3$ to $C$. An upper bound for $C_3$ is obtained from assessment of $\pi(x/p_1;q,\overline{p_1})$ with large values of $q$ by lemma 2.1. Remark that lemma 2.1 states for the case when $a$ is a fixed integer, but in our case the role of $a$ is taken by $\overline{p_1}$, which can increase as exponent of $x$. By the similar argument of C. Hooley \cite{Hooley73}, one should obtain the same result for the case large $a$. The reason is that the constants, which appeared  in Hooley's proof, implied by the O-notation depend on the given non-zero a integer $a$ like divisor function $\tau(a)\ll x^{\gamma}$ and $0<\gamma-$ arbitrarily small number. The condition $p_1 \leq x^{1-4\theta/(3-\varepsilon)} $ implies that $q\leq x^{\theta}<\left(\frac{x}{p_1}\right)^{3/4-\varepsilon}$. We divide interval $\sqrt{x/p_1}<q\leq x^{\theta}  $ into sub-intervals of form $X<q\leq rX$ and $\sqrt{x/p_1}<X\leq x^{\theta}$. Here $r\in (1,2)$ is an arbitrarily number. Let $G$ be defined from conditions $$
		r^G\sqrt{\frac{x}{p_1}}>x^{\theta}\geq r^{G-1}\sqrt{\frac{x}{p_1}}\Leftrightarrow  G>\left(\theta\log{x}-0.5\log{(x/p_1)}\right)(\log{r})^{-1}\geq G-1.$$
		Then, we write
		$$C_3\leq \sum_{p_1\leq p_1 \leq x^{1-4\theta/(3-\varepsilon)}}\sum_{g=1}^G\sum_{r^{g-1}\sqrt{\frac{x}{p_1}} <q\leq r^{g}\sqrt{\frac{x}{p_1}}}(\log{q})\pi\left(\frac{x}{p_1};q,\overline{p_1}\right)$$
		$$\leq \sum_{p_1\leq p_1 \leq x^{1-4\theta/(3-\varepsilon)}}\sum_{g=1}^G \frac{(4+\eta)(r-1)x}{p_1\log{x/p_1}} \leq x(4+\eta_1)\frac{(r-1)}{\log{r}}\sum_{p_1 \leq x^{1-4\theta/(3-\varepsilon)}}\frac{\theta\log{x}}{p_1\log{(x/p_1)}}-\frac{1}{2p_1}$$
		$$\leq (4+\eta_2)(\theta-0.5)\frac{(r-1)}{\log{r}}x\log{\log{x}}.$$ Let $r$ tend to $1$ we infer that
		\begin{equation}
			\label{C3}
			C_3\leq (4+\eta_2)(\theta-0.5)x\log{\log{x}},
		\end{equation}
		where $\eta_2>0$ satisfies condition such that limit $\eta_2$ as $x$ approaches infinity equal to zero.
		
		The sum $C_4$ must be evaluated more carefully because of large values, which can be taken by $q$ and $p_1.$ More specifically, $x/(qp_1)\leq 1$ when $q\asymp x^{\theta}$ and $p_1\asymp \sqrt{x}$, therefore, assessment of Selberg's type and Brun-Titchmarsh's type can not be used. To evaluate $C_4$ we are going to use lemma 2.2. Let us write
		$$C_4\leq \sum_{x^{3/8-0.5\varepsilon}<q\leq x^{\theta}}\log{q}\sum_{\substack{x^{2-8\theta/(3-\varepsilon)}<n\leq x\\n-\text{semi-prime}}}1.$$ Put $N=x^{2-8\theta/(3-\varepsilon)}$ and define $G$ from inequalities $$ 2^GN>x\geq 2^{G-1}N\Leftrightarrow G-1\leq \left(\frac{8\theta}{3-\varepsilon}-1\right)\frac{\log{x}}{\log{2}}<G.$$ Then, we have $$ C_4\leq \sum_{x^{3/8-0.5\varepsilon}<q\leq x^{\theta}}\log{q}\sum_{g=1}^G\sum_{\substack{2^{g-1}N<n\leq 2^gN\\n-\text{semi-prime}}}1$$
		$$\ll \sum_{x^{3/8-0.5\varepsilon}<q\leq x^{\theta}}\log{q}\sum_{g=1}^G \frac{N\log{\log{N}}}{\log{N}} \ll \sum_{x^{3/8-0.5\varepsilon}<q\leq x^{\theta}}(\log{q})x^{2-8\theta/(3-\varepsilon)}\log{\log{x}}$$
		$$\ll x^{2+\theta-8\theta/(3-\varepsilon)}\log{\log{x}}.$$
		From condition $0.6<\theta$ we have
		\begin{equation}
			\label{C4}
			C_4\ll x.
		\end{equation}
		
		From \ref{C2}, \ref{C3} and \ref{C4} we have 
		$$C_5\geq (0.5-(4+\eta_2)(\theta-0.5))x\log{\log{x}}+O(x)\geq (2.5-4\theta-\eta_3)x\log{\log{x}}>0$$ for all $\theta<0.625$. Our theorem is deduced from definition of $C_5$.
	\end{proof} 
	\section{Conclusion}
	It worth to pointing out that the limit $\theta<0.625$ comes from the C. Hooley's assessment (lemma 2.1), which connected to distribution of primes in arithmetic progressions $a, \pmod q $ for large values of $q$. Since semi-primes appear more often in set of natural numbers, an analogue of Hooley's lemma for semi-primes in arithmetic progressions for large module $q$ should be able to allow one obtain more exactly limit for $\theta.$ We remark also  that our result can be compared with earlier Hooley's bound $P^+(p+a)>x^{0.625}$.
	
\end{document}